\documentclass[12pt]{amsart}

\newtheorem{theorem}{Theorem}%[section]
\newtheorem{lemma}[theorem]{Lemma}
\newtheorem{proposition}[theorem]{Proposition}
\newtheorem{corollary}[theorem]{Corollary}

\theoremstyle{definition}

\newtheorem{remark}[theorem]{Remark}

\usepackage{amscd,amssymb}
\begin{document}

\title[GK-dimension of $2\times 2$ generic Lie matrices]
{GK-dimension of the Lie algebra\\
of generic $2\times 2$ matrices}

\author[V. Drensky, P. Koshlukov, and G. Machado]
{Vesselin Drensky, Plamen Koshlukov, Gustavo Grings Machado}
\address{Institute of Mathematics and Informatics,
Bulgarian Academy of Sciences,
Acad. G. Bonchev Str., Block 8, 1113 Sofia, Bulgaria}
\email{drensky@math.bas.bg}

\address{Department of Mathematics, IMECC, UNICAMP,
S\'ergio Buarque de Holanda 651, Campinas,
SP 13083-859, Brazil}
\email{plamen@ime.unicamp.br}

\address{Department of Mathematics, CCNE, UFSM,
Faixa de Camobi, Campus UFSM,
Santa Maria, RS 97105-900, Brazil}
\email{grings@smail.ufsm.br}

\thanks
{This project was partially supported by Grant I02/18
``Computational and Combinatorial Methods
in Algebra and Applications''
of the Bulgarian National Science Fund.
P. Koshlukov acknowledges partial support by FAPESP grant No. 2014/08608-0, and by CNPq grants No. 304003/2011-5 and No. 480139/2012-1.}


\begin{abstract}
Recently Machado and Koshlukov have computed the Gelfand-Kirillov dimension
of the relatively free algebra $F_m=F_m(\text{\rm var}(sl_2(K)))$ of rank $m$
in the variety of algebras generated by the three-dimensional simple Lie algebra
$sl_2(K)$ over an infinite field $K$ of characteristic different from 2.
They have shown that $\text{\rm GKdim}(F_m)=3(m-1)$.
The algebra $F_m$ is isomorphic to the Lie algebra generated by $m$ generic
$2\times 2$ matrices. Now we give a new proof for
$\text{\rm GKdim}(F_m)$ using classical results of Procesi and Razmyslov
combined with the observation that the commutator ideal of $F_m$ is a module
of the center of the associative algebra generated by $m$ generic traceless $2\times 2$ matrices.
\end{abstract}

\maketitle

\section{Introduction}

Let $R$ be a (not necessarily associative) algebra generated by $m$
elements $r_1,\ldots,r_m$ over a field $K$ and
let $V_n$ be the vector subspace of $R$ spanned by all products
$r_{i_1}\cdots r_{i_k}$, $k\leq n$. The growth function of $R$
with respect to the given system of generators is
\[
g_R(n)=\dim(V_n),\quad n\geq 0.
\]
The Gelfand-Kirillov dimension of $R$
is defined as
\[
\text{\rm GKdim}(R)=\limsup_{n\to\infty}\log_n(g_R(n)).
\]
It does not depend on the choice of
the generators of $R$. See the book \cite{KL} for a background on GKdim. If the algebra $R$ is graded,
\[
R=\bigoplus_{n\geq 0}R^{(n)},
\]
where $R^{(n)}$ is the homogeneous component of degree $n$ of $R$, then
the Hilbert series of $R$ is the formal power series
\[
H(R,t)=\sum_{n\geq 0}\dim(R^{(n)})t^n.
\]
If $R$ is generated by its homogeneous elements of first degree, then
its growth function is
\[
g_R(n)=\sum_{l=0}^n\dim(R^{(l)}).
\]
In the general case, if $R$ is a graded algebra generated by a finite
system of (homogeneous)
%%added (homogeneous)
elements of arbitrary degree,
its Gelfand-Kirillov dimension can be expressed using again its Hilbert series
as
\[
\text{\rm GKdim}(R)=\limsup_{n\to\infty}\log_n
\left(\sum_{l=0}^n\dim(R^{(l)})\right).
\]

When studying varieties of $K$-algebras $\mathfrak V$,
all information for the $m$-generated algebras in $\mathfrak V$
is carried by the relatively free algebra $F_m({\mathfrak V})$
of rank $m$ in $\mathfrak V$. When the base field $K$
is of characteristic 0, a lot is known for
the Gelfand-Kirillov dimension of relatively free associative algebras,
see the book \cite{KL}, the survey article \cite{D2},
or the paper \cite{MK}.
In particular, $\text{\rm GKdim}(F_m({\mathfrak V}))$ is an integer
for all proper varieties of associative algebras.
Almost nothing
is known for relatively free Lie algebras. Using the bases of free
nilpotent-by-abelian Lie algebras given by Shmelkin \cite{Sh},
it is easy to see that
\[
\text{\rm GKdim}(F_m({\mathfrak N}_c{\mathfrak A}))
=\text{\rm GKdim}(L_m/(L'_m)^{c+1})=mc,
\]
where $m>1$ and $L_m$ is the free $m$-generated Lie algebra.
Together with free nilpotent Lie algebras where the Gelfand-Kirillov
dimension is equal to 0, these are the only free polynilpotent Lie algebras
of finite Gelfand-Kirillov dimension, see Petrogradsky \cite{Pe}.

Recently Machado and Koshlukov \cite{MK} have computed the Gelfand-Kirillov dimension
of the relatively free algebra $F_m=F_m(\text{\rm var}(sl_2(K)))$ of rank $m>2$
in the variety of algebras generated by the three-dimensional simple Lie algebra
$sl_2(K)$ over an infinite field $K$ of characteristic different from 2.
They have shown that $\text{\rm GKdim}(F_m)=3(m-1)$. Their proof is based on
a careful analysis of the explicit expression of the Hilbert series of
$F_m$ obtained by Drensky \cite{D1}. The case $m=2$ was handled before
by Bahturin \cite{B} who showed that $\text{\rm GKdim}(F_2)=3$.
The algebra $F_m$ is isomorphic to the Lie algebra generated by $m$ generic traceless
%%added traceless
$2\times 2$ matrices. The purpose of our paper is to give a new
proof for $\text{\rm GKdim}(F_m)$ using
classical results of Procesi \cite{P1, P2} on Gelfand-Kirillov dimension
of the algebra of generic matrices  and Razmyslov \cite{R}
on the weak polynomial identities of matrices,
combined with the observation that the commutator ideal of $F_m$ is a module
over the center of the associative algebra generated by $m$ generic traceless $2\times 2$ matrices.
We believe that the present approach is more adequate
for generalizations for other finite dimensional
simple Lie algebras than the approach in \cite{MK}.

\section{The proof}

The following statement and its corollary are folklorely known.
We include the proof for self-completeness of the exposition
and also because we were not able to find an explicit reference.

\begin{lemma}\label{multiplicity of the pole}
Let $R$ be a finitely generated graded algebra with Hilbert series
of the form
\[
H(R,t)=h(t)\prod_{i=1}^s\frac{1}{(1-t^{d_i})},
\]
where $h(t)\in {\mathbb C}[t]$ is a polynomial
and the $d_i$'s are positive integers.
Then the Gelfand-Kirillov dimension of $R$ is equal to the multiplicity
of $1$ as a pole of $H(R,t)$.
\end{lemma}

\begin{proof}
It is sufficient to consider the case when
$R$ is not finite dimensional and hence
its Hilbert series has a nontrivial denominator.
Let $d$ be the least common multiple of the degrees $d_i$. Then
\[
H(R,t)=\sum_{n\geq 0}a_nt^n=f(t)+\sum_{p=1}^k\sum_{q=0}^{d-1}
\frac{\alpha_{pq}}{(1-\omega_qt)^p}
\]
\[
=f(t)+\sum_{n\geq 0}\left(\sum_{p=1}^k\binom{n+p-1}{p-1}
\sum_{q=0}^{d-1}\alpha_{pq}\omega_q^n\right)t^n,
\]
where $f(t)\in {\mathbb C}[t]$, $\alpha_{pq}\in{\mathbb C}$,
$\omega_0=1,\omega_1,\ldots,\omega_{d-1}$ are the $d$-th roots of 1,
and at least one of the coefficients $\alpha_{kq}$ is different from zero.
Since $\omega_q^d=1$, the sequences
\[
\beta_{pn}=\sum_{q=0}^{d-1}\alpha_{pq}\omega_q^n,
\quad p=1,\ldots,k,
\]
are periodic with period $d$ and for $n$ large enough the coefficients
$a_n$ of the Hilbert series $H(R,t)$ are bounded by polynomials
of degree $k-1$ in $n$. Hence the sequence
\[
\sum_{l=0}^na_l=\sum_{l=0}^n\dim(R^{(l)})
\]
needed for the definition of the Gelfand-Kirillov dimension of $R$
is bounded by a polynomial of degree $k$ in $n$ and
\[
\text{\rm GKdim}(R)\leq k.
\]
The asymptotics of the coefficients $a_n$ of
\[
H(R,t)=f(t)+\sum_{n\geq 0}\left(\sum_{p=1}^k\binom{n+p-1}{p-1}
\beta_{pn}\right)t^n,
\]
is determined by $\beta_{kn}$. Since $a_n$ are positive integers, we derive
that the periodic sequence $\beta_{kn}$, $n=0,1,2,\ldots$, consists of
nonnegative reals and at least one of them is positive. Since $\omega_q^d=1$, if $\omega_q\not=1$, then
$1+\omega_q+\omega_q^2+\cdots+\omega_q^{d-1}=0$. Hence
\[
0<\sum_{l=0}^{d-1}\beta_{k,dn+l}
=\sum_{l=0}^{d-1}\sum_{q=0}^{d-1}\alpha_{kq}\omega_q^{dn+l}
\]
\[
=\sum_{q=0}^{d-1}\alpha_{kq}\sum_{l=0}^{d-1}\omega_q^l
=d\alpha_{k0}.
\]
Therefore $\alpha_{k0}>0$. We consider the partial sum
$p_{dn}=a_0+a_1+\cdots+a_{dn}$ of the coefficients of the
Hilbert series $H(R,t)$. Its asymptotics is determined by
\[
\tilde{p}_{dn}=\sum_{c=0}^{dn}\binom{c+k-1}{k-1}\beta_{kc}
\approx\frac{1}{(k-1)!}\sum_{c=0}^{dn}c^{k-1}\beta_{kc}
\]
\[
\approx\frac{1}{(k-1)!}\sum_{e=0}^n(ed)^{k-1}
\sum_{l=0}^{d-1}\beta_{k,ed+l}
=\frac{d\alpha_{k0}}{(k-1)!}\sum_{e=0}^n(ed)^{k-1}
\]
and this is a polynomial of degree $k$ in $n$. Hence
\[
\text{\rm GKdim}(R)=\limsup_{n\to\infty}\log_n\left(\sum_{l=0}^na_l\right)
\geq\limsup_{n\to\infty}\log_n\left(\sum_{c=0}^{dn}a_c\right)
\]
\[
=\limsup_{n\to\infty}\log_{dn}(p_{dn})
=\limsup_{n\to\infty}\log_{dn}(\tilde{p}_{dn})=k
\]
which, together with the opposite inequality
$\text{\rm GKdim}(R)\leq k$, completes the proof.
\end{proof}

\begin{corollary}\label{GKdim of commutative algebras}
Let $R$ be a finitely generated graded algebra and let $C$ be a
finitely generated graded subalgebra of the center of $R$ such that
$R$ is a finitely generated $C$-module.
Then the Gelfand-Kirillov dimension of $R$ is equal to the multiplicity
of $1$ as a pole of $H(R,t)$.
\end{corollary}

\begin{proof}
By the Hilbert-Serre theorem (see e.g., \cite{AM}), the Hilbert series of
any finitely generated graded module $M$ over a finitely generated graded
commutative algebra $C$ is of the form
\[
H(M,t)=h(t)\prod_{i=1}^k\frac{1}{(1-t^{d_i})},
\quad h(t)\in {\mathbb C}[t],d_i>0.
\]
Hence the proof follows immediately from Lemma
\ref{multiplicity of the pole}.
\end{proof}

In the sequel we assume that the base field $K$ is of characteristic 0.
Let
\[
\Omega_{km}=K[Y_{km}]=K[y_{pq}^{(i)}\mid p,q=1,\ldots,k,i=1,\ldots,m]
\]
be the algebra of polynomials in $k^2m$ commuting variables and let
\[
y_i=(y_{pq}^{(i)}),\quad i=1,\ldots,m,
\]
be $m$ generic $k\times k$ matrices. We consider the following algebras:

\noindent $R_{km}$ -- the generic matrix
%%changed trace
algebra. This is the subalgebra
generated by $y_1,\ldots,y_m$ of the associative $k\times k$
matrix algebra $M_k(\Omega_{km})$ with entries from $\Omega_{km}$.

\noindent $C_{km}$ -- the pure trace algebra. This is the subalgebra
of $\Omega_{km}$ generated by the traces of the products,
$\text{tr}(y_{i_1}\cdots y_{i_l})$. We embed $C_{km}$ in $M_k(\Omega_{km})$
by $f(Y_{km})\to f(Y_{km})I_k$, where $I_k$ is the identity matrix.

\noindent $T_{km}$ -- the mixed trace algebra. This is the subalgebra
of $M_k(\Omega_{km})$ generated by $R_{km}$ and $C_{km}$.

For a background on generic matrices see e.g., \cite{P2} or \cite{DFo}.
Below we summarize the results we need.

\begin{proposition}\label{properties of generic matrices}
Let $k,m\geq 2$. Then:

{\rm (i)} The mixed trace algebra $T_{km}$ has no zero divisors;

{\rm (ii)} The pure trace algebra $C_{km}$ coincides with
the center of $T_{km}$. It is finitely generated
and $T_{km}$ is a finitely generated $C_{km}$-module;

{\rm (iii) \cite{Ki,P1}}
\[
\text{\rm GKdim}(T_{km})=\text{\rm GKdim}(C_{km}) =\text{\rm GKdim}(R_{km}) =k^2(m-1)+1.
%%added =\text{\rm GKdim}(R_{km})
%%formally no need for it - but it is important as a result
\]
\end{proposition}

Further, we consider the generic traceless $k\times k$ matrices
\[
z_i=(z_{pq}^{(i)})=y_i-\frac{1}{k}\text{tr}(y_i)I_k,\quad i=1,\ldots,m,
\]
and the subalgebra $W_{km}$ of $T_{km}$
generated by $z_1,\ldots,z_m$, the subalgebra $C_{km}^{(0)}$ of
$C_{km}$ generated by the traces of the products,
%%%traces of the products: it's better than products of traces, although understandable
$\text{tr}(z_{i_1}\cdots z_{i_l})$, and the subalgebra $T_{km}^{(0)}$
of $T_{km}$ generated by $W_{km}$ and $C_{km}^{(0)}$. Finally, let
$L_{km}$ be the Lie subalgebra of $W_{km}$ generated by $z_1,\ldots,z_m$.

\begin{proposition}\label{properties of traceless matrices}
Let $k,m\geq 2$. Then

{\rm (i) (Procesi \cite{P3})}
\[
T_{km}\cong K[\text{\rm tr}(y_1),\ldots,\text{\rm tr}(y_m)]
\otimes_KT_{km}^{(0)},
\]
\[
C_{km}\cong K[\text{\rm tr}(y_1),\ldots,\text{\rm tr}(y_m)]
\otimes_KC_{km}^{(0)};
\]

{\rm (ii) (Razmyslov \cite{R})}
\[
W_{km}\cong K\langle x_1,\ldots,x_m\rangle/\text{\rm Id}(M_k(K),sl_k(K))
%%corrected sl_2 to sl_k
\]
where $\text{\rm Id}(M_k(K),sl_k(K))$ is the ideal of all
weak polynomial identities in $m$ variables
for the pair $(M_k(K),sl_k(K))$, i.e., the polynomials
in the free associative algebra $K\langle x_1,\ldots,x_m\rangle$
which vanish when evaluated on $sl_k(K)$ considered as a subspace in
$M_k(K)$.

{\rm (iii) (Razmyslov \cite{R})} The Lie algebra $L_{km}$ is isomorphic to
the relatively free algebra $F_m(\text{\rm var}(sl_k)(K))$ in the
variety of Lie algebras generated by $sl_k(K)$.
\end{proposition}

\begin{corollary}\label{GKdim of mixed algebra generated by traceless
matrices}
For $k,m\geq 2$
\[
\text{\rm GKdim}(T_{km}^{(0)})=\text{\rm GKdim}(C_{km}^{(0)})=(k^2-1)(m-1).
\]
\end{corollary}

\begin{proof}
The algebras $T_{km}$ and $C_{km}$ satisfy the conditions of Corollary
\ref{GKdim of commutative algebras}. Hence the multiplicity of 1 as a pole
of the Hilbert series of $T_{km}$ and $C_{km}$ is equal to their
Gelfand-Kirillov dimension $k^2(m-1)+1$ (see  Proposition
\ref{properties of generic matrices} (iii)).
Proposition \ref{properties of traceless matrices} (i) gives that
\[
H(T_{km},t)=H(K[\text{tr}(y_1),\ldots,\text{tr}(y_m)],t)H(T_{km}^{(0)},t)
=\frac{1}{(1-t)^m}H(T_{km}^{(0)},t),
\]
\[
H(C_{km},t)=\frac{1}{(1-t)^m}H(C_{km}^{(0)},t).
\]
Hence the multiplicity of 1 as a pole of $H(T_{km}^{(0)},t)$
and $H(C_{km}^{(0)},t)$ is equal to $(k^2(m-1)+1)-m=(k^2-1)(m-1)$.
Both algebras $T_{km}^{(0)}$ and $C_{km}^{(0)}$ are finitely generated
and graded. Hence the proof follows from Corollary
\ref{GKdim of commutative algebras}.
\end{proof}

Now we shall summarize the information for $2\times 2$ generic matrices.

\begin{proposition}\label{generic 2 x 2 matrices}
Let $k=2$ and $m\geq 2$. Then:

{\rm (i) (Sibirskii \cite{Si})}
The trace polynomials
\[
\text{\rm tr}(y_i),\, i = 1,\ldots,m,\quad
\text{\rm tr}(y_iy_j),\, 1 \leq i\leq j \leq m,
\]
\[
\text{\rm tr}(y_{i_1}y_{i_2}y_{i_3}),\, 1 \leq i_1<i_2<i_3\leq m,
\]
form a minimal system of generators of $C_{2m}$.

{\rm (ii) (Procesi \cite{P3})} The algebras $T_{2m}^{(0)}$ and
$W_{2m}$ coincide. The algebra $C_{2m}^{(0)}$ is generated by
\[
\text{\rm tr}(z_iz_j),\, 1 \leq i\leq j \leq m,\quad
\text{\rm tr}(z_{i_1}z_{i_2}z_{i_3}),\, 1 \leq i_1<i_2<i_3\leq m,
\]
which belong to $W_{2m}$.

{\rm (iii) (Drensky \cite{D3})} The algebra $C_{2m}^{(0)}$ is generated by
\[
z_i^2,\,i=1,\ldots,m,\quad z_iz_j+z_jz_i,\,1 \leq i\leq j \leq m,
\]
\[
s_3(z_{i_1},z_{i_2},z_{i_3}),\, 1 \leq i_1<i_2<i_3\leq m,
\]
where
\[
s_3(x_1,x_2,x_3)=\sum_{\sigma\in S_3}\text{\rm sign}(\sigma)
x_{\sigma(1)}x_{\sigma(2)}x_{\sigma(3)}
\]
is the standard polynomial of degree $3$.
\end{proposition}

\begin{proof} We shall present the proof of (ii) and (iii)
as a consequence of (i). Clearly $C_{2m}^{(0)}$ is generated by
$\text{\rm tr}(z_iz_j)$, $1 \leq i\leq j \leq m$, and
$\text{\rm tr}(z_{i_1}z_{i_2}z_{i_3})$, $1\leq i_1<i_2<i_3\leq m$.
Now the proof of (ii) and (iii) follows immediately from the equalities
in $T_{2m}^{(0)}$
\[
\text{\rm tr}(z_1^2)=2z_1^2,\quad
\text{\rm tr}(z_1z_2)=z_1z_2+z_2z_1,
\]
\[
\text{\rm tr}(z_1z_2z_3)=\frac{1}{3}s_3(z_1,z_2,z_3)
\]
which may be checked by direct verification.
\end{proof}

\begin{lemma}\label{the commutator ideal is a module}
The commutator ideal $L'_{2m}$ is a $C_{2m}^{(0)}$-module.
\end{lemma}

\begin{proof}
The following equalities which can be verified directly
hold in $W_{2m}$:
\[
[z_1,z_2]z_3^2=\frac{1}{4}([z_1,z_2,z_3,z_3]-[[z_1,z_3],[z_2,z_3]]),
\]
\[
[z_1,z_2](z_3z_4+z_4z_3)=\frac{1}{4}([z_1,z_2,z_3,z_4]+[z_1,z_2,z_4,z_3]
\]
\[
-[[z_1,z_3],[z_2,z_4]]-[[z_1,z_4],[z_2,z_3]]),
\]
\[
z_4s_3(z_1,z_2,z_3)=\frac{3}{8}\sum_{\sigma\in S_3}\text{sign}(\sigma)
[z_4,z_{\sigma(1)},z_{\sigma(2)},z_{\sigma(3)}].
\]
The elements of the commutator ideal are linear combinations of
(left normed) commutators $u_i=[z_{i_1},z_{i_2},\ldots,z_{i_n}]$. If
$v$ is a generator of $C_{2m}^{(0)}$, then
\[
u_iv=[z_{i_1},z_{i_2},\ldots,z_{i_n}]v=[[z_{i_1},z_{i_2}]v,\ldots,z_{i_n}]
\]
and the above equalities guarantee that $u_iv$ is a linear combination
of commutators, i.e., belongs to $L'_{2m}$ again. Hence
$L'_{2m}C_{2m}^{(0)}\subset L'_{2m}$.
\end{proof}

\begin{remark}
It is known that $W_{2m}$ is a $C_{2m}^{(0)}$-module generated by
$1$, $z_i$, $i=1,\ldots,m$, and $[z_i,z_j]$, $1\leq i<j\leq m$.
Using the equality
\[
[z_1,z_2,z_3]=2(z_1(z_2z_3+z_3z_2)-z_2(z_1z_3+z_3z_1)),
\]
as in the proof of Lemma \ref{the commutator ideal is a module}
we can show that $L'_{2m}$ is a $C_{2m}^{(0)}$-module generated by
all commutators $[z_i,z_j]$ and $[z_{i_1},z_{i_2},z_{i_3}]$.
For $m=2$, the commutator ideal $L'_{22}$ is a free
$C_{22}^{(0)}$-module generated by $[z_1,z_2]$, $[z_1,z_2,z_1]$, $[z_1,z_2,z_2]$,
see \cite{DFi}.
\end{remark}

The proof of the following theorem established in \cite{MK}
is the main result of our paper.

\begin{theorem}
Let $K$ be a field of characteristic $0$ and let $L_{2m}$ be the Lie
algebra generated by $m$ generic traceless $2\times 2$ matrices,
$m\geq 2$. Then
\[
\text{\rm GKdim}(L_{2m})
=\text{\rm GKdim}(F_m(\text{\rm var}(sl_2(K)))=3(m-1).
\]
\end{theorem}

\begin{proof}
Let
\[
H(C_{2m}^{(0)},t)=\sum_{n\geq 0}c_nt^n,\quad
H(L_{2m},t)=\sum_{n\geq 1}l_nt^n,\quad
H(W_{2m},t)=\sum_{n\geq 1}w_nt^n
\]
be the Hilbert series of $C_{2m}^{(0)}$, $L_{2m}$, and $W_{2m}$,
respectively.
Since the algebra $L_{2m}$ is finitely generated,
its Gelfand-Kirillov dimension is
\[
\text{\rm GKdim}(L_{2m})=\limsup_{n\to\infty}\log_n
\left(\sum_{k=1}^nl_k\right).
\]
The algebra $W_{2m}$ has no zero divisors and hence
$[z_1,z_2]C_{2m}^{(0)}\subset L'_{2m}\subset L_{2m}$
is a free $C_{2m}^{(0)}$-module. Therefore
\[
\sum_{k=0}^{n-2}c_k\leq \sum_{k=1}^nl_k\leq \sum_{k=0}^nw_k,
\]
which implies that
\[
3(m-1)=\text{\rm GKdim}(C^{(0)}_{2m})\leq \text{\rm GKdim}(L_{2m})
\leq \text{\rm GKdim}(W_{2m})=3(m-1).
\]
\end{proof}

\begin{remark}
As in \cite{MK}, the formula for the Gelfand-Kirillov dimension of
$F_m(\text{\rm var}(sl_2(K)))$ obtained in characteristic 0 holds
also for any infinite field $K$ of characteristic different from 2.
\end{remark}

\begin{remark}
In characteristic 2, the algebra $sl_2(K)$ is nilpotent of class 2
and hence $F_m(\text{\rm var}(sl_2(K)))$ is isomorphic to the free nilpotent
of class 2 Lie algebra $F_m({\mathfrak N}_2)$ which is finite dimensional.
Therefore $\text{\rm GKdim}(F_m(\text{\rm var}(sl_2(K))))=0$.
When $K$ is an infinite field of characteristic 2,
a much more interesting object is the relatively free algebra
$F_m(\text{\rm var}(M_2(K)^{(-)}))$ of the variety generated by the
$2\times 2$ matrix algebra $M_2(K)$ considered as a Lie algebra.
Vaughan-Lee \cite{VL} showed that the algebra $M_2(K)^{(-)}$ does not have
a finite basis of its polynomial identities. (It is easy to see that
the four-dimensional Lie algebra constructed in \cite{VL} is isomorphic to
$M_2(K)^{(-)}$.) The algebra $M_2(K)^{(-)}$ satisfies
the center-by-metabelian polynomial identity
\[
[[[x_1,x_2],[x_3,x_4]],x_5]=0.
\]
It is well known that the free center-by-metabelian Lie algebra
$F_m([{\mathfrak A}^2,{\mathfrak E}])$ over any field $K$
is spanned by
\[
[x_{i_1},x_{i_2},x_{i_3},\ldots,x_{i_n}],\quad
[[x_{i_1},x_{i_2},x_{i_3},\ldots,x_{i_n}],[x_{i_{n+1}},x_{i_{n+2}}]],
\]
where $i_1>i_2\leq i_3\leq \cdots\leq i_n$ and the commutators are
left normed, e.g., $[x_1,x_2,x_3]=[[x_1,x_2],x_3]$. (A basis of
$F_m([{\mathfrak A}^2,{\mathfrak E}])$ is given by Kuzmin \cite{Ku}.)
Since the commutators $[x_{i_1},x_{i_2},x_{i_3},\ldots,x_{i_n}]$
form a basis of the free metabelian Lie algebra
$F_m({\mathfrak A}^2)$ and are linearly independent in
$F_m(\text{\rm var}(M_2(K)^{(-)}))$, we obtain immediately that
\[
\text{\rm GKdim}(F_m(\text{\rm var}(M_2(K)^{(-)})))=m,\quad m>1.
\]
In characteristic 2 there is another three-dimensional simple Lie algebra
which is an analogue of the Lie algebra of the three-dimensional
real vector space with the vector multiplication.
It is interesting to see whether this algebra has a finite basis of its
polynomial identities (probably not) and, when the field is infinite,
to compute the Gelfand-Kirillov dimension of the corresponding
relatively free algebras.
\end{remark}

\section*{Acknowledgements}
The authors are grateful to the anonymous referees for the useful suggestions for improving of the exposition.

\end{document}